\documentclass[12pt, a4paper,leqno]{article}

\usepackage[english]{babel}
\usepackage[utf8]{inputenc}
\usepackage[T1]{fontenc}

\usepackage{graphicx}

\usepackage{paralist}

\usepackage{mathtools}
\usepackage{amssymb}
\usepackage{amsthm}
\usepackage{accents}
\usepackage{mathrsfs}
\usepackage{dsfont}

\newtheorem{thm}{Theorem}[section]
\newtheorem{lem}[thm]{Lemma}
\newtheorem{prop}[thm]{Proposition}

\newtheorem{rem}[thm]{Remark}
\newtheorem{cor}[thm]{Corollary}
\theoremstyle{definition}

\newcommand{\norm}[2]{||#1||_{#2}}
\newcommand{\skp}[2]{\langle #1 \rangle_{#2}}
\newcommand{\RR}{\mathbb{R}}
\newcommand{\NN}{\mathbb{N}}

\newcommand{\oG}{\overline{G}}
\newcommand{\oT}{\overline{T}}
\newcommand{\ogamma}{\overline{\gamma}}
\newcommand{\oGamma}{\overline{\Gamma}}
\newcommand{\Wop}{W^{out}_p}
\newcommand{\tV}{\tilde{V}}
\newcommand{\oc}{\overline{c}}
\newcommand{\orp}{\overline{r}}
\newcommand{\oQ}{\overline{Q}}
\usepackage{verbatim}

\numberwithin{equation}{section}

\title{Formation of singularities in solutions to nonlinear hyperbolic systems with general sources}
\author{Johannes Bärlin}

\begin{document}
\maketitle

\begin{abstract}
We consider nonlinear hyperbolic systems with a general source and prove that for appropriately chosen smooth initial data the lifespan of the associated $C^1$-solution $u$ cannot be infinite. We employ ideas of F.~John \cite{joh74} and L. Hörmander \cite{hor87} to show that the derivative $u_x$ of $u$ becomes unbounded in finite time.
\end{abstract}

This preprint has not undergone peer review (when applicable) or any post-submission
improvements or corrections. The Version of Record of this article is published in \textit{Nonlinear Analysis: Real World Applications} as \cite{jb23}, and
is available online at https://doi.org/10.1016/j.nonrwa.2023.103901.

\section{Introduction}

In this paper we consider quasi-linear hyperbolic systems of the form
\begin{align}
\label{conservationlawwithsource}
u_t + a(u) u_x = g(u)
\end{align}
on $(0,\infty) \times \RR$ where $a: \Omega \mapsto \RR^{N \times N}$ and $g: \Omega \mapsto \RR^N$ are smooth functions with $N \in \NN$, $\Omega \subset \RR^N$ open and $0 \in \Omega$, satisfying the following three assumptions:
\begin{itemize}
\item[A1] $a(0)$ has only real, simple eigenvalues;
\item[A2] There is $p \in \lbrace 1,\ldots,N \rbrace$ such that the $p$-th eigenvalue of $a$ is genuinely nonlinear at $0$;
\item[A3] $g(0) = 0$.
\end{itemize}
By adapting ideas of John and Hörmander \cite{joh74,hor87} we will show:
\begin{thm}
\label{mainthm}
If A1, A2, A3 hold then there exist smooth compactly supported functions $u^0:\RR \mapsto \Omega$ such that while remaining bounded, the unique $C^1$-solution $u$ of \eqref{conservationlawwithsource} with data $u(0,\cdot) = u^0$ exists only for finite time.
\end{thm}

The question whether solutions to hyperbolic systems develop singularities over time is classical. For $g\equiv 0$, it has been answered affirmatively for general small and compactly supported data under the assumption of genuine non-linearity by P. Lax in \cite{lax64} for $N =1,2,$ and for general $N \in \NN$ by F. John in \cite{joh74}, with a notable generalization by T.-P.~Liu in \cite{liu79}.

If however $g \not \equiv 0$, the situation can be dichotomous in the sense that
\begin{itemize}
\item[(i)] for small data a smooth solution to \eqref{conservationlawwithsource} exists for all time while
\item[(ii)] for some large initial data the smooth solution of \eqref{conservationlawwithsource} must develop a singularity.
\end{itemize}
While property (i) is known to hold for quite general classes notably of relaxation systems \cite{zen99,mz02,yon04,ky04,bhn07,bz11} (cf.~also, e.g., \cite{rug04,ll09,hf22}), blow-up results (ii) have been shown in particular cases by ad hoc constructions, see e.g.~\cite{ll09,hw18,hrw22,jb22}.

The finding of the present paper shows that, independently of the nature of the source $g$, it is the genuine non-linearity of the differential operator $\partial_t + a(\cdot) \partial_x$ that by itself entails the breakdown of certain smooth solutions. The result applies to all systems considered in \cite{zen99,mz02,yon04,ky04,bhn07,bz11} if the space dimension is equal to one and the assumptions A1-A3 are satisfied. It is sharp in the sense that there exist significant relaxation systems with linear principal part $\partial_t + a \partial_x$ $-$ and nonlinear equilibrium system $-$ whose solutions remain smooth for all times, see e.g.~\cite{ll09}, Sec.~4.

\section{Preliminaries}
We start with some notation. Let $T> 0, m,n \in \NN$ and $k \in \lbrace 1,2\rbrace$. The space $C^\infty_c(\RR;\RR^m)$ contains all infinitely many times continuously differentiable functions $f: \RR \rightarrow \RR^m$ with compact support $\textnormal{supp}\, f \subset \RR$. We denote by $C([0,T] \times \RR;\RR^m)$ the space of continuous functions $f:[0,T] \times \RR \mapsto \RR^m$. A function $f:[0,T] \times \RR \rightarrow \RR^m$ is an element of the space $C^k([0,T] \times \RR; \RR^m)$ if for all $\alpha \in \NN^2_0, |\alpha| \leq k,$ the partial derivative $\partial^\alpha f$ exists on $(0,T) \times \RR$ and is equal to a function $f_\alpha \in C([0,T] \times \RR;\RR^m)$ there. Here $\partial^\alpha$ is the usual multi-index notation for $\partial_t^{\alpha_1} \partial_x^{\alpha_2}$. We write $f \in C^k_c([0,T]\times \RR;\RR^m)$ if the support of $f$ is compact. Similar definitions are used for $C([0,T) \times \RR;\RR^m)$ and $C^k([0,T) \times \RR;\RR^m)$. We will drop the reference to $\RR^m$ in the following since it should be clear from the context where a considered function takes its values. Finally if $f:\RR^n \rightarrow \RR^m$ is a differentiable function then $Df$ denotes its Jacobi-matrix. If $n=1$ we will also write $f^\prime$.

Fix a $p \in \lbrace 1,\ldots, N \rbrace$ satisfying assumption A2. By A1 there are $\delta > 0$ and smooth functions
\begin{align*}
\lambda_i: & B_{2\delta}(0) \subset \RR^N \rightarrow \RR,\\
r_i: & B_{2\delta}(0) \subset \RR^N \rightarrow \RR^{N\times 1},\\
l_i: & B_{2\delta}(0) \subset \RR^N \rightarrow \RR^{1 \times N},
\end{align*}
for all $i \in \lbrace 1, \ldots, N \rbrace$ which satisfy for all $u \in B_{2 \delta}(0)$ and all $i,j$
\begin{align*}
l_i(u) a(u) = \lambda_i(u) l_i(u), \; a(u) r_i(u) = \lambda_i(u) r_i(u), \; \textnormal{and} \; \lambda_i(u) < \lambda_j(u) \;\textnormal{if} \; i < j.
\end{align*}
We normalize the above eigenvectors of $a$ such that for all $i,j$
\begin{align*}
|l_i| = 1 \; \textnormal{and} \; l_i r_j = \delta_{ij},
\end{align*}
and such that by Assumption A2 it holds
\begin{align}
\label{gnl}
\skp{D\lambda_p(0), r_p(0)}{} < 0.
\end{align}
Set
\begin{align}
\label{defclambda}
c_{\lambda} := \min_{i \neq p} \inf_{u \in B_\delta(0)} |\lambda_i(u) - \lambda_p(u)| > 0.
\end{align}
Fix $T> 0$ and a solution 
\begin{align}
\label{regsol}
u \in C^2([0,T] \times \RR)
\end{align}
of \eqref{conservationlawwithsource} with
\begin{align}
\label{ubound}
|u(t,x)| \leq \delta.
\end{align}
By differentiation of \eqref{conservationlawwithsource} the derivative of $u$
\begin{align*}
w(t,x) := u_x(t,x)
\end{align*}
satisfies the equation
\begin{align}
\label{equationw}
w_t + a(u) w_x = Dg(u) w + \gamma(u,w)
\end{align}
where $\gamma$ is a quadratic form in $w$. Decompose $w$ with respect to $\lbrace r_j(u) \rbrace_j$ to obtain
\begin{align*}
w(t,x) = \sum_j w_j(t,x) r_j(u(t,x)).
\end{align*}
The component functions $w_i$ satisfy certain partial differential equations along the $i$-th characteristic curves. The latter are determined by the ordinary differential equations
\begin{align}
\label{ODEchar}
\frac{d}{dt} X_i(t;x) = \lambda_i(u(t,X_i(t;x))), \; X_i(0;x) = x,
\end{align}
which may be solved under the above hypothesis on $[0,T]$ for all $x \in \RR$. The $i$-th characteristic curve $C_i(x)$ starting at $(0,x)$ is given by the trace of
\begin{align*}
t \mapsto (t,X_i(t;x)).
\end{align*}
For given $i$ and $(t,x) \in [0,T] \times \RR$ let $z_i=z_i(t,x)$ be the unique element of $\RR$ such that the $i$-th characteristic curve starting at $(0,z_i)$ runs through $(t,x)$, i.e.
\begin{align*}
(t,x) \in C_i(z_i).
\end{align*}
If $f: [0,T] \times \RR \rightarrow \RR$ is a differentiable function then let $L_i f$ denote the derivative of $f$ with respect to $t$ along the $i$-th characteristic:
\begin{align*}
L_i f(t,x) = f_t(t,x) + \lambda_i(u(t,x)) f_x(t,x).
\end{align*}
From \eqref{equationw} one obtains a system of partial differential equations:
\begin{align}
\label{charequw}
L_i w_i = \sum_{j,k} \gamma_{ijk}(u) w_jw_k + \sum_k G_{ik}(u) w_k
\end{align}
where $\gamma_{ijk}$ are given in \cite{joh74} and below; they satisfy for all $i,j$
\begin{align}
\label{gammarelation}
\gamma_{ijj}(u) = - \delta_{ij} \skp{D\lambda_i(u),r_i(u)}{}.
\end{align}
The functions $G_{ik}$ will be discussed below, too. There are constants $\ogamma, \oG > 0$ such that
\begin{align*}
\max_i \sup_{u \in B_\delta(0)} \sum_{j,k} |\gamma_{ijk}(u)| & \leq \ogamma\\
\max_i \sup_{u \in B_\delta(0)} \sum_k |G_{ik}(u)| & \leq \oG.
\end{align*}
Following \cite{hor87} we calculate the differential of $w_i(dx - \lambda_i(u) dt)$ and obtain by \eqref{charequw}
\begin{align}
\label{dwi}
d(w_i(dx - \lambda_i(u) dt) = \left(\sum_{jk} \Gamma_{ijk}(u) w_j w_k + \sum_k G_{ik}(u) w_k \right) dt \wedge dx
\end{align}
where the $\Gamma_{ijk}$ are like in \cite{hor87}, in particular for all $i,j$
\begin{align}
\label{Gammarelation}
\Gamma_{ijj}(u) = 0.
\end{align}
We find a constant $\oGamma>0$ such that
\begin{align*}
\sup_i \sup_{u \in B_\delta(0)} \sum_{j,k} |\Gamma_{ijk}(u)| \leq \oGamma.
\end{align*}
The reader may find expressions and relations for the coefficient functions $c_{ijk}, \gamma_{ijk}, \Gamma_{ijk}$ and $G_{ik}$ at the end of this section. We end our considerations on a solution \eqref{regsol}, \eqref{ubound} by the following statement which is a straightforward extension of Lemma 1.2.2 in \cite{hor87} to systems with sources.
\begin{lem}
\label{lem1}
Let $T> 0$ und let $i \in \lbrace 1, \ldots N \rbrace$. Suppose $u \in C^2([0,T]\times \RR)$ solves \eqref{conservationlawwithsource} with $|u(t,x)| \leq \delta$. Let $\tau$ be a $C^1$-arc in $[0,T] \times \RR$ intersecting the $i$-th characteristic curves transversally, and let $A_i(\tau)$ be the open region bounded by $\tau$, the orbits of $L_i$ through the end points of $\tau$ and an interval $\tau_{i0}$ where $\lbrace t = 0 \rbrace$. Then
\begin{align}
\label{estimatewitau}
\int\limits_\tau |w_i(dx & - \lambda_i(u)dt)|\\
\nonumber & \leq \int\limits_{\tau_{i0}} |w_i| dx + \int\limits_{A_i(\tau)} |\sum_{j,k} \Gamma_{ijk}(u) w_j w_k + \sum_k G_{ik}(u) w_k)| dxdt.
\end{align}
\end{lem}
\begin{proof}
The proof is essentially the same as in \cite{hor87} using \eqref{dwi} and Stokes' formula. Let us just comment on the case when $w_i$ does not have constant sign on $\tau$. Then one considers the (relatively) open set
\begin{align*}
\tilde{\tau} := \tau \cap \lbrace w_i \neq 0 \rbrace.
\end{align*}
Use a compact exhaustion $\lbrace \tau_j \rbrace_{j \in \NN}, \tau_j \subset \textnormal{int}\,\tau_{j+1}, \cup_j \tau_j = \tilde{\tau},$  of $\tilde{\tau}$ and note that each $\tau_j$ consists of finitely many singletons and pathwise connected components. Then apply the considerations holding when $w_i$ has constant sign on $\tau$ to each of these components.
\end{proof}
Assumption A3 $g(0) = 0$ implies a propagation of support result for solutions to \eqref{conservationlawwithsource} which we report in the next lemma.
\begin{lem}
\label{lem2}
Let $T >0$. Suppose $u \in C^1([0,T] \times \RR)$ is a solution of \eqref{conservationlawwithsource} with
\begin{align*}
\textnormal{supp}\, u(0,\cdot) \subset [a,b]
\end{align*}
for some $a <b$. Then for all $t \in [0,T]$ it holds
\begin{align*}
\textnormal{supp}\, u(t,\cdot) \subset [a + \lambda_1(0)t,b+\lambda_N(0)t].
\end{align*}
\end{lem}
\begin{proof}
Consider all characteristics emanating, at $t=0$, from $\RR \backslash (a,b)$.
\end{proof}
We close this section with an existence result for \eqref{conservationlawwithsource} with smooth compactly supported data. For a function $f:\RR \rightarrow \RR^N$ set
\begin{align*}
\norm{f}{\infty} := \sup\limits_{x \in \RR} |f(x)|.
\end{align*}
\begin{lem}
\label{lemexist}
Let $c>0$ and $\tilde{\delta} < \delta$, and consider a compactly supported $C^2$-function $u^0: \RR \rightarrow B_{\tilde{\delta}}(0)$ with $\norm{u^0_x}{\infty} \leq c$. Then there are a time $T=T(\tilde{\delta},c) \in(0,\infty)$ and a unique solution $u \in C^2([0,T] \times \RR)$ of \eqref{conservationlawwithsource} with initial data $u(0,\cdot) = u^0$. For all $\alpha \in \NN^2_0, |\alpha| \leq 2$, the derivatives $\partial^\alpha u$ of $u$ are bounded and for all $(t,x) \in [0,T] \times \RR$ it holds $|u(t,x)| < \delta$.
\end{lem}
\begin{proof}
Adapt the proof of Theorem 1.2.5.~in \cite{hor87} using assumption A3. The source term $g$ does not cause any problems since it is a lower order term. Also note \cite{lax73} and \cite{rau3} on how to deal with strictly hyperbolic systems in one space dimension.
\end{proof}

We finally give some relations and expressions concerning certain coefficient functions appearing in the above analysis. From \cite{joh74} we have
\begin{align*}
c_{ijk}(u) = \left(\frac{d}{ds} (l_i(u) a(u+sr_k(u)) r_j(u) \right)_{s=0}
\end{align*}
and
\begin{align*}
&  \;\;\;\;\;\;\;\;\;\;\;\; \gamma_{ijk} = \gamma_{ikj}, \; \gamma_{iii} = - c_{iii},\\
2 \gamma_{iik} & = - c_{iik} - c_{iki} + \sum_{j, j \neq i} \frac{\lambda_i - \lambda_k}{\lambda_j - \lambda_i} c_{ijk}\skp{l_j,l_i}{} \; \textnormal{if} \; k \neq i,\\
2 \gamma_{ijk} & = -\frac{\lambda_j-\lambda_k}{\lambda_j-\lambda_i} c_{ijk} - \frac{\lambda_k - \lambda_j}{\lambda_k - \lambda_i} c_{ikj} \; \textnormal{if} \; j \neq i, k \neq i.
\end{align*}
The coefficients $\Gamma_{ijk}$ are determined by the requirement \cite{hor87}
\begin{align*}
\sum_{j,k} \gamma_{ijk}(u) w_j w_k + \sum_k w_i w_k \skp{D\lambda_i(u),r_k(u)}{} = \sum_{j,k} \Gamma_{ijk}(u) w_j w_k 
\end{align*}
for all $w \in \RR^N$. An explicit expression for $G_{ik}$ may be obtained from the requirement that for all $w \in \RR^N$ it holds
\begin{align*}
\sum_k G_{ik}(u) w_k = \sum_k l_i Dg(u) r_k w_k + \sum_{j,k, j \neq i} \frac{1}{\lambda_j - \lambda_i} c_{ijk} (l_k g(u))(\skp{l_j,l_i}{} w_i - w_j).
\end{align*}

\section{Initial data and smallness properties}

We begin with the construction of the special initial data that will lead to the formation of a singularity. Let $U:I \rightarrow B_\delta(0)$ be an integral curve of $r_p$ through $0$ defined on an open interval $I \subset \RR$ with $0 \in I$:
\begin{align*}
U^\prime(\xi) = r_p(U(\xi)) \; (\xi \in I), \; U(0) = 0.
\end{align*}
Let $\alpha \in C^\infty_c(\RR)$ with $\alpha(\RR) \subset I, \textnormal{supp}\, \alpha \subset (-1/2,1/2)$ and $\max\limits_x |\alpha^\prime| = \max\limits_x \alpha^\prime > 0$. For $\varepsilon \in (0,1]$ define
\begin{align*}
u^{0\varepsilon}(x) = U(\varepsilon \alpha(x)) \; (x \in \RR).
\end{align*}
Then $u^{0\varepsilon} \in C^\infty_c(\RR)$.
\begin{rem}
For $g \equiv 0$, these data give rise to a breaking wave. The idea of using the same data for the system with non-vanishing source is that the associated blow-up of the quantity $w_p$ persists if the source term is small.
\end{rem}

The purpose of this section consists in showing the following

\begin{lem}
\label{lem3}
For any fixed $\oT \in (0,\infty)$ there exist constants $c_J,c_M,c_S,c_V > 0$ and $\nu > 0$ such that if for $\varepsilon, \kappa  \in (0,\nu]$ and 
\begin{align*}
0<T \leq \oT \varepsilon^{-1}
\end{align*}
a function $u \in C^2([0,T] \times \RR)$ solves the Cauchy problem
\begin{align}
\label{CPrescaled}
\left\lbrace
\begin{aligned}
u_t + a(u) u_x & = \varepsilon \kappa g(u),\\ u(0,x) & = u^{0\varepsilon}(x).
\end{aligned}\right. 
\end{align}
then for all $t \in [0,T]$ the quantities
\begin{align*}
J(t) & := \sup\limits_{s \in [0,t]} \int_{a_p(s)}^{b_p(s)} |w_p(s,x)| dx,\\
M(t) & := \sup\limits_{(s,x)\in [0,t] \times \RR} |u(s,x)|,\\
S(t) & := \sup\limits_{s \in [0,t]} b_p(s) - a_p(s),\\
\tilde{V}(t) & := \max\limits_{i\neq p} \sup\limits_{(s,x)\in [0,t] \times \RR} |w_i(s,x)|,\\
W^{out}_p(t) & := \sup\limits_{\substack{(s,x) \in [0,t] \times \RR, \\(s,x) \not \in R_p(t)}} |w_p(s,x)|,\\
V(t) & := \Wop(t) + \tV(t),
\end{align*}
satisfy
\begin{align}
\label{cJ} J(t) & < c_J \varepsilon\\
\label{cM} M(t) & < c_M \varepsilon < \delta\\
\label{cS} S(t) & < c_S\\
\label{cV} V(t) & < c_V \varepsilon^2.
\end{align}
Here we have set
\begin{align*}
b_p(t) & := X_p(t;1/2)\\
a_p(t) & := X_p(t;-1/2)\\
R_p(t) & := \lbrace (s,x) \in [0,t] \times \RR| \; x \in [a_p(s),b_p(s)] \rbrace.
\end{align*}
\end{lem}
\begin{rem}
Since we expect the solution $u$ to be of the same size as $u^{0\varepsilon}$ we find that $J$ and $M$ should be comparable to $u^{0\varepsilon}$, too, hence of order $\varepsilon$. The quantities $\tV$ and $\Wop$ are initially equal to zero, and thinking of the idea of perturbing the simple wave, we hope that the interaction between $w_i, i \neq p,$ and $w_p$ is weak, and that $w_p$ along the $p$-th characteristics outside $R_p$ remains small compared to $w_p$ inside $R_p$ by the propagation properties of the hyperbolic differential operator in \eqref{CPrescaled}.
\end{rem}
\begin{proof}
Let $\oT > 0$. For
\begin{align*}
\oc &:= \max \left\lbrace 1, \max\limits_i \sup\limits_{u \in B_\delta(0)} \sum_{j,k} |c_{ijk}(u)| \right\rbrace\\
\orp & := \sup\limits_{u \in B_\delta(0)} \sum\limits_k |r_k(u)|
\end{align*}
set
\begin{align*}
c_J & := 2 \max_x \alpha^\prime(x)\\
c_V & := 2 \oG c_\lambda^{-1} c_J(1+\oG\,\oT)\\
c_S & := 2 \oc(1+c_J \oT)\\
c_M & := 2 \orp (c_J + c_V(1 + \oT(\lambda_N(0) - \lambda_1(0))).
\end{align*}
Take $\nu$ so small that
\begin{align}
\label{cond1}
c_M \nu < \delta
\end{align}
holds. Let $\varepsilon, \kappa  \in (0,\nu]$ with $\nu$ to be determined further later on, and let $T \in (0,\oT \varepsilon^{-1}]$. Suppose we are given a solution $u \in C^2([0,T] \times \RR)$ of the Cauchy problem \ref{CPrescaled}. By Lemma \ref{lem2} we find that $u \in C^2_c([0,T] \times \RR)$. At $t = 0$ it holds
\begin{align}
\label{M0}
|u(0,x)| = |u^{0\varepsilon}(x)| = |\int\limits^x_{-\infty} \varepsilon \alpha^\prime(y) r_p(U(\varepsilon \alpha(y))dy| \leq \varepsilon \orp c_J \leq \frac{c_M}{2} \varepsilon < c_M \varepsilon
\end{align}
Because $u \in C^2_c$ the function $M$ is continuous on $[0,T]$ and we find a $t_1 \in (0,T]$ such that \eqref{cM} holds for all $t \in [0,t_1]$ since the stronger inequality \eqref{M0} holds at $t=0$. Hence on $[0,t_1]$ the functions $J,S, V,\tV,\Wop$ are well-defined, continuous and increasing. We will show that the set
\begin{align*}
E_1 := \lbrace t \in [0,T]|\; \eqref{cJ}-\eqref{cV}\; \textnormal{hold on} \; [0,t] \rbrace
\end{align*}
is non-empty, open and closed, hence equal to $[0,T]$. We have by 
\begin{align*}
u_x(0,x) = \varepsilon \alpha^\prime(x) r_p(u(0,x)) = w_p(0,x) r_p(u(0,x))
\end{align*}
that initially the waves $w_j$ satisfy
\begin{align}
\label{initialwi} w_i(0,x) & = 0 \;\;\;\;\;\;\;\; \textnormal{if}\; i \neq p,\\
\label{initialwp} w_p(0,x) & = \varepsilon \alpha^\prime(x).
\end{align}
By \eqref{initialwi}, \eqref{initialwp} and the compact support of $\alpha$ in $(-1/2,1/2)$ we find $0 \in E_1$ since
\begin{align*}
J(0) & \leq \varepsilon \max_x \alpha^\prime = \varepsilon \frac{c_J}{2}\\
S(0) & = 1 \leq \frac{c_S}{2}\\
V(0) & = 0 \leq \frac{c_V}{2}.
\end{align*}
By continuity and monotonicity of $J,M,S$ and $V$ the set $E_1$ is open in $[0,T]$. It remains to show that $E_1$ is closed which is by monotonicity equivalent to:
\begin{align*}
\textnormal{If for} \; T^\prime \in (0,T] \; \textnormal{\eqref{cJ}-\eqref{cV} hold on} \; [0,T^\prime) \; \textnormal{then \eqref{cJ}-\eqref{cV} hold at} \; t = T^\prime.
\end{align*}
Suppose for $T^\prime \in (0,T]$ we have $[0,T^\prime) \subset E_1$ and let $t \in [0,T^\prime)$. Note that $M(T^\prime) < \delta$ by \eqref{cM}. For $s \in [0,t]$ it holds (compare \cite{joh74} p.~394)
\begin{align*}
\frac{d}{ds} (b_p(s)-a_p(s)) & = \int\limits^{b_p(s)}_{a_p(s)} \sum_k c_{ppk}(u(s,x)) w_k(s,x) dx\\
& \leq \oc (\tV(s)S(s) + J(s)).
\end{align*}
Integrating the last relation over $[0,t]$ we find by $t \leq \oT \varepsilon^{-1}$ and $[0,t] \subset E_1$ that
\begin{align*}
S(t) & \leq S(0) + \oc (\tV(t) S(t) t + J(t) t)\\
& \leq \oc (1+c_J\oT) + \oc \oT c_V \varepsilon S(t).
\end{align*}
If $\nu$ satisfies
\begin{align}
\label{cond2}
1- \oc \oT c_V \nu \geq \frac{2}{3}
\end{align}
then
\begin{align*}
S(t) \leq \frac{3}{2} \oc(1+c_J \oT) < c_S
\end{align*}
which shows $S(T^\prime) < c_S$.

To estimate $J(t)$ we use Lemma \ref{lem1} with $\tau = \lbrace s \rbrace \times [a_p(s),b_p(s)]$ for $s \in [0,t]$ to obtain
\begin{align*}
\int\limits_{a_p(s)}^{b_p(s)} |w_p(s,x)| dx \leq J(0) + \int\limits_{A_p(\tau)} |\sum_{j,k} \Gamma_{pjk}(u) w_j w_k + \varepsilon \kappa \sum_k G_{ik}(u) w_k| dx dt.
\end{align*}
By \eqref{Gammarelation} we see that the terms involving $\Gamma_{pjk}(u) w_j w_k$ are essentially linear and we get
\begin{align*}
J(t) & \leq J(0)+ \oGamma \lbrack\tV(t)^2 S(t) t + \tV(t) J(t) t\rbrack + \varepsilon \kappa \oG\lbrack\tV(t)S(t) t + J(t)t\rbrack\\
& \leq \varepsilon \frac{c_J}{2} + \varepsilon^3 \oGamma c_V^2c_S \oT + \varepsilon^2 \kappa \oG c_V c_s \oT + J(t) \lbrack \varepsilon \oGamma c_V \oT + \kappa \oG \, \oT \rbrack.
\end{align*}
If $\nu$ satisfies
\begin{align}
\label{cond3} \nu^2 (\oGamma c_V^2c_S \oT + \oG c_V c_s \oT) & \leq \frac{1}{8} c_J,\\
\label{cond4} 1-\nu(\oGamma c_V \oT + \oG \, \oT) & \geq \frac{5}{6}
\end{align}
then
\begin{align*}
J(t) \leq \frac{3}{4} c_J \varepsilon < c_J \varepsilon
\end{align*}
implying $J(T^\prime) < c_J \varepsilon$.

We estimate $u(s,x)$ for all $(s,x) \in [0,t] \times \RR$ using Lemma \ref{lem2} by
\begin{align*}
|u(s,x)| & = |\int\limits_{-\infty}^x u_x(s,y) dy|\\
& \leq \int\limits_{-1/2 +\lambda_1(0)s}^{1/2 +\lambda_N(0)s} |u_x(s,x)|dy\\
& \leq \orp (J(t) + V(t)(1 + t(\lambda_N(0) - \lambda_1(0)))\\
& \leq \varepsilon \orp (c_J + c_V(1+\oT(\lambda_N(0) - \lambda_1(0)))\\
& = \varepsilon \frac{c_M}{2}
\end{align*}
yielding $M(T^\prime) \leq \varepsilon \frac{c_M}{2} < \varepsilon c_M < \delta$.

The estimates for $V(t)$ are the longest ones. We begin with an estimate of $\Wop(t)$: Let $(s,x) \in [0,t] \times \RR$ with $(s,x) \not \in R_p(t)$. Then for $z_p = z_p(s,x)$ and all $\sigma \in [0,t]$ we find
\begin{align}
\label{Rpout}
(\sigma, X_p(\sigma,z_p)) \not \in R_p(t)
\end{align}
because the $p$-th characteristic curves $C_p(y)$ may not cross for different choices of $y \in \RR$ since they are defined through solutions of \eqref{ODEchar} with $i = p$. Furthermore, by \eqref{initialwp} it holds
\begin{align}
\label{initialwp2}
w_p(0,z_p) = 0.
\end{align}
Integrating $L_p w_p$ along $C_p(z_p)$ and using \eqref{charequw}, \eqref{Rpout} and \eqref{initialwp2} yields
\begin{align*}
|w_p(s,x)| & = | \int\limits_0^s \sum_{j,k} \gamma_{pjk}(u) w_j w_k + \varepsilon \kappa \sum_k G_{ik}(u) w_k d\sigma|\\
& \leq t \ogamma(V(t) + \varepsilon \kappa \oG)V(t),
\end{align*}
and therefore by $[0,t] \subset E_1$
\begin{align}
\label{Wopestimate}
\Wop(t) \leq \oT(\ogamma c_V \varepsilon + \oG \kappa) V(t).
\end{align}
We turn to the estimate for $\tV(t)$. Let $i \neq p$ and $(s,x) \in [0,t] \times \RR$. For $z_i = z_i(s,x)$ define
\begin{align*}
\omega_i(s,x) := \lbrace \sigma \in [0,t]| \; (\sigma,X_i(\sigma;z_i)) \in R_p(t) \rbrace .
\end{align*}
Then $\omega_i(s,x)$ is a (possibly empty) closed interval in $[0,t]$ and the trace of 
\begin{align*}
\omega_i(s,x) \ni \sigma \mapsto (\sigma,X_i(\sigma;z_i))
\end{align*}
defines a smooth arc $\tau_i(s,x)$ in $[0,t] \times \RR$ intersecting the $p$-th characteristic curves transversally by the separation of eigenvalues of $a$ \eqref{defclambda}. We have $w_i(0,z_i) = 0$ by \eqref{initialwi} and integrating the equation \eqref{charequw} for $L_i w_i$ along $C_i(z_i)$ gives
\begin{align*}
|w_i(s,x)| = |\int\limits_0^s \sum_{j,k} \gamma_{ijk}(u) w_j w_k + \varepsilon \kappa \sum_k G_{ik}(u) w_k d\sigma|.
\end{align*}
Using the important relations \eqref{gammarelation} for the coefficients $\gamma_{ijk}$ yields
\begin{align}
\label{Vestimate1}
|w_i(s,x)| \leq (\ogamma V(t) + \varepsilon \kappa \oG) \left\lbrack t V(t) + \int_0^t |w_p(\sigma,X_i(\sigma;z_i))| d\sigma. \right\rbrack
\end{align}
We need to estimate the integral of $|w_p|$ along $C_i(z_i)$. To this end we make use of Lemma \ref{lem1} and \eqref{defclambda}:
\begin{align}
\label{Vestimate2}
\int_0^t |w_p(\sigma,X_i&(\sigma;z))| d\sigma\\
\nonumber \leq & \Wop(t)t + \int\limits_{w_i(s,x)} |w_p| \frac{\lambda_i(u)-\lambda_p(u)}{\lambda_i(u)-\lambda_p(u)} d\sigma\\
\nonumber \leq & V(t)t + c_\lambda^{-1} \int\limits_{\tau_i(s,x)} |w_p(dx-\lambda_p(u)dt)|\\
\nonumber \leq & V(t)t + c_\lambda^{-1} J(0)\\
\nonumber & +  c_\lambda^{-1} \int\limits_{A_p(\tau_i)} |\sum_{j,k} \Gamma_{pjk}(u) w_j w_k + \varepsilon \kappa \sum_k G_{ik}(u) w_k| dxdt\\
\nonumber \leq & V(t) t + c_\lambda^{-1}\lbrack c_J \varepsilon + t(\oGamma V(t) + \varepsilon \kappa \oG) (V(t) S(t) + J(t))\rbrack
\end{align}
where the relations \eqref{Gammarelation} for $\Gamma_{pjk}$ were used in the last step.
Combining the estimates \eqref{Wopestimate}-\eqref{Vestimate2} we find a function $Q=Q(\varepsilon, \kappa)$ given by
\begin{align*}
Q(\varepsilon,\kappa) = & 3 \oT (\ogamma c_V \varepsilon + \oG \kappa) + \ogamma \varepsilon c_\lambda^{-1} \lbrack c_J + \oT (\oGamma c_V \varepsilon + \oG \kappa)(c_V c_S \varepsilon + c_J)\rbrack\\
& + \oG \kappa c_{\lambda}^{-1} \oT \varepsilon(\oGamma(c_Vc_S\varepsilon + c_J) + \oG c_S \kappa),
\end{align*}
and a constant $\oQ := Q(1,1)$ such that for all $\varepsilon,\kappa \in (0,\nu]$ we have $Q(\varepsilon, \kappa) \leq \oQ\nu$ and
\begin{align*}
V(t) & \leq Q(\varepsilon, \kappa) V(t) + \kappa \oG c_\lambda^{-1}c_J(1+ \kappa \oG\, \oT)\varepsilon^2\\
& \leq \oQ\nu V(t) + \frac{c_V}{2}\varepsilon^2.
\end{align*}
If $\nu$ satisfies
\begin{align}
\label{cond5}
1 - \nu \oQ \leq \frac{3}{4}
\end{align}
then
\begin{align*}
V(t) \leq \frac{2}{3} c_V \varepsilon^2 < c_V \varepsilon^2.
\end{align*}
Hence $T^\prime \in E_1$ follows if $\nu$ satisfies the conditions \eqref{cond1}, \eqref{cond2}, \eqref{cond3}, \eqref{cond4} and \eqref{cond5}.
\end{proof}
As a first application of Lemma \ref{lem3} we prove that any $C^1$-solution of \eqref{CPrescaled} on $[0,T]$ must be a $C^2$-solution of \eqref{CPrescaled} on $[0,T]$ if $T$ is smaller than a certain $\varepsilon$-depended bound.
\begin{lem}
\label{lem4}
Let $\oT \in (0,\infty)$ and choose $\nu=\nu(\oT)$ such that the conclusions of Lemma \ref{lem3} hold. Let $\varepsilon, \kappa \in [0,\nu]$. If $0<T\leq \oT \varepsilon^{-1}$ and $u \in C^1([0,T] \times \RR)$ solves \eqref{CPrescaled} then $u \in C^2([0,T] \times \RR)$.
\end{lem}
\begin{proof}
Let $\varepsilon,\kappa \in [0,\nu]$, $T \in (0,\oT \varepsilon^{-1}]$ and suppose $u \in C^1([0,T] \times \RR)$ solves \eqref{CPrescaled}. By Lemma \ref{lem2} we have $u \in C^1_c([0,T] \times \RR)$ and there is a constant $c_{u_x}> 0$ such that for all $(t,x) \in [0,T] \times \RR$
\begin{align*}
|u_x(t,x)| \leq c_{u_x}.
\end{align*}
We argue by continuous induction. Set
\begin{align*}
E_3 := \lbrace T^\prime \in [0,T]| \; u \in C^2([0,T^\prime] \times \RR) \rbrace.
\end{align*}
We need to show that $E_3$ is non-empty, relatively open and closed. By Lemma \ref{lemexist} there is a $C^2$-solution on $[0,T^\prime]$ of \eqref{CPrescaled} for some $T^\prime > 0$ since $u^{0\varepsilon} \in C^2_c(\RR)$ and $\norm{u^{0\varepsilon}}{\infty} < \varepsilon c_M < \delta$ where $c_M$ is chosen like in Lemma \ref{lem3}. By Lemma \ref{lem2} this solution has to be equal to $u$ on $[0,T^\prime]$, i.e. $T^\prime \in E_3$.

Choose $\tilde{\nu} > \nu$ with $\tilde{\nu}c_M< \delta$. Let $\Delta t = \Delta t(\tilde{\nu} c_M, c_{u_x})> 0$ be a time of existence for a $C^2$-solution of 
\begin{align}
\label{epskapequation}
u_t + a(u) u_x = \varepsilon \kappa g(u)
\end{align}
with initial data satisfying appropriate bounds as given by Lemma \ref{lemexist}. Let $T^\prime \in E_3$. Because $u(T^\prime,\cdot) \in C^2_c(\RR), \norm{u(T^\prime,\cdot)}{\infty} < \varepsilon c_M$ and $\norm{u_x(T^\prime,\cdot)}{\infty} \leq c_{u_x}$ there is a solution $\tilde{u} \in  C^2([T^\prime,T^\prime + \Delta t] \times \RR)$ of \eqref{epskapequation} with $\tilde{u}(T^\prime,x) = u(T^\prime,x)$ for $x \in \RR$. Uniqueness again yields $u \equiv \tilde{u}$ on $[T^\prime,T^\prime + \Delta t] \cap [T^\prime,T]$. We conclude that $E_3$ is open in $[0,T]$.

Let $T^\prime \in [0,T]$. We show that if $[0,T^\prime) \subset E_3$ than $T^\prime \in E_3$. This implies the closedness of $E_3$. Because $E_3 \not = \emptyset$ the case $T^\prime = 0$ follows immediately. Therefore suppose $T^\prime >0$ and $[0,T^\prime) \subset E_3$. Then $u \in C^2([0,T^\prime) \times \RR)$. Set $\tilde{t} := T^\prime - \frac{\min \lbrace T^\prime, \Delta t \rbrace}{2}$. By $T^\prime \leq T \leq \oT \varepsilon^{-1}$, Lemma \ref{lem3} and continuity of $u$ on $[0,T]$ we find for all $(t,x) \in [0,T^\prime] \times \RR$ that
\begin{align*}
|u(t,x)| \leq \nu c_M < \tilde{\nu} c_M < \delta.
\end{align*}
Because $|u_x(\tilde{t},x)| \leq c_{u_x}$, too,  there is a $C^2$-solution of \eqref{epskapequation} on $[\tilde{t},\tilde{t} + \Delta t]$ with initial data $u(\tilde{t},\cdot)$. By $\tilde{t} < T^\prime < \tilde{t} + \Delta t$ we find $u \in C^2$ on $[0,T^\prime]$, i.e. $T^\prime \in E_3$.
\end{proof}

\section{The gradient blow-up}

Set
\begin{align*}
\oT := \frac{4}{\gamma_{ppp}(0) \max \alpha^\prime}.
\end{align*}
We are now ready to prove the key finding of this paper:
\begin{prop}
\label{thm1}
There exist $\nu \in (0,1]$ such that if for $\varepsilon, \kappa  \in (0,\nu]$ and $T>0$ a function $u \in C^2([0,T]\times \RR)$ solves \eqref{CPrescaled} with $u(0,x) = u^{0\varepsilon}(x)$ then $T$ must satisfy
\begin{align}
\label{Tepsiloin}
T < T_\varepsilon := \frac{3}{4}\oT \varepsilon^{-1}.
\end{align}
\end{prop}

\begin{proof} Let $z \in \RR$ such that $\alpha^\prime(z) = \max\limits_x \alpha^\prime$ and let $\varepsilon, \kappa \in (0,\nu]$ for $\nu=\nu(\oT)$ like in Lemma \ref{lem3}. Suppose $u \in C^2([0,T_\varepsilon] \times \RR)$ solves \eqref{CPrescaled}. Then the statement of Lemma \ref{lem3} applies to $u$ since $T_\varepsilon \leq \oT \varepsilon^{-1}$. For $t \in [0,T_\varepsilon]$ define
\begin{align*}
W(t) := w_p(t,X_p(t;z)).
\end{align*}
Then $W(0) = \varepsilon \alpha^\prime(z)$ and $W$ satisfies \eqref{charequw}:
\begin{align*}
W^\prime = \sum_{j,k} \gamma_{pjk}(u) w_j w_k + \varepsilon \kappa \sum_k G_{ik}(u) w_k.
\end{align*}
Set
\begin{align*}
c_W := 2 \frac{c_V}{\alpha^\prime(z)}.
\end{align*}
Let $\nu$ be so small such that for all $\varepsilon \in (0,\nu]$ and for all $(t,x) \in [0,T_\varepsilon]\times \RR$
\begin{align}
\label{cond6}
\gamma_{ppp}(u(t,x)) > \frac{1}{2} \gamma_{ppp}(0)
\end{align}
which is possible by $|u(t,x)| < c_M \varepsilon$ (see \eqref{cM}). We find
\begin{align*}
W^\prime > \frac{1}{2}\gamma_{ppp}(0) W^2 - \ogamma(V^2 + V |W|) - \varepsilon \kappa \oG(V + |W|).
\end{align*}
We claim that for $\nu$ small enough the relations
\begin{align}
\label{lowerboundW}
W(t) & > \frac{W(0)}{2} > 0\\
\label{upperboundV}
V(t) & < \varepsilon c_W W(t)
\end{align}
hold for all $t \in [0,T_\varepsilon]$. Like in the proof of Lemma \ref{lem3} we proceed by continuous induction. Set
\begin{align*}
E_2 := \lbrace t \in [0,T_\varepsilon]| \; \textnormal{\eqref{lowerboundW} and \eqref{upperboundV} hold on} \; [0,t] \rbrace.
\end{align*}
Again $E_2$ will be non-empty, open and closed in $[0,T_\varepsilon]$. Since
\begin{align*}
W(0) & = \varepsilon \alpha^\prime(z) > \frac{W(0)}{2} > 0,\\
V(0) & = 0 < \varepsilon \frac{c_W}{2} W(0) < \varepsilon c_W W(0)
\end{align*}
we find by continuity of $W$ and $V$ a $t_2>0$ with $[0,t_2] \subset E_2$.

Let $t \in E_2 \backslash \lbrace 0 \rbrace$. We show that if $\varepsilon$ and $\kappa$ are small enough then for all $s \in [0,t]$ the stronger inequalities
\begin{align}
\label{strongerlowerboundW} W(s) & \geq W(0)\\
\label{strongerupperboundV} V(s) & \leq \varepsilon \frac{c_W}{2} W(s)
\end{align}
hold. By continuity this implies that $E_2$ is open in $[0,T_\varepsilon]$, and by the procedure in the proof of Lemma \ref{lem3} it also implies its closedness in $[0,T_\varepsilon]$. For $s \in [0,t]$ \eqref{upperboundV} gives
\begin{align*}
W^\prime(s) > \left(\frac{1}{2}\gamma_{ppp}(0) - \ogamma c_W \varepsilon (1+ c_W \varepsilon)\right)W(s)^2 - \varepsilon \kappa \oG(1+ c_W \varepsilon)W(s).
\end{align*}
If we choose $\nu$ so small that
\begin{align}
\label{cond7} 
\frac{1}{2}\gamma_{ppp}(0) - 2 \ogamma c_W \nu & > \frac{3}{8} \gamma_{ppp}(0),\\
\label{cond8}\nu c_W + 1 & < 2
\end{align}
then
\begin{align}
\label{derW}
W^\prime(s) > \frac{3}{8} \gamma_{ppp}(0)W(s)^2 - 2 \varepsilon \kappa \oG W(s).
\end{align}
If $\nu$ satisfies
\begin{align}
\label{cond9}
\frac{3}{16} \gamma_{ppp}(0)\alpha^\prime(z) - 2 \nu\oG > \frac{1}{8} \gamma_{ppp}(0)\alpha^\prime(z)
\end{align}
then we obtain by \eqref{lowerboundW}
\begin{align*}
W^\prime(s) & > (\frac{3}{16} \gamma_{ppp}(0)\alpha^\prime(z) - 2 \kappa \oG) \varepsilon W(s)\\
& > \frac{1}{8} \gamma_{ppp}(0)\alpha^\prime(z) \varepsilon W(s)\\
& > 0.
\end{align*}
Hence $W$ increases on $[0,t]$ and $W(s) \geq W(0)$ holds for all $s \in [0,t]$. Furthermore, by \eqref{cV} we find for all $s \in [0,t]$
\begin{align*}
V(s) < c_V \varepsilon^2 = \frac{c_V}{\alpha^\prime(z)} \varepsilon W(0) \leq  \frac{c_V}{\alpha^\prime(z)} \varepsilon W(s) = \frac{c_W}{2} \varepsilon W(s)
\end{align*}
which concludes the proof of the inequalities \eqref{strongerlowerboundW} and \eqref{strongerupperboundV} if $\nu$ satisfies the conditions \eqref{cond7}, \eqref{cond8} and \eqref{cond9}.

Finally let $y$ be the solution of
\begin{align}
\label{ycompequa}
y^\prime = \frac{3}{8} \gamma_{ppp}(0)y^2 - 2 \varepsilon \kappa \oG y, \; y(0) = W(0).
\end{align}
Then by Gronwall's Lemma and \eqref{derW} we have
\begin{align*}
y(s) \leq W(s)
\end{align*}
as long as $y$ and $W$ exist. Solving \eqref{ycompequa} explicitly and using \eqref{cond9} we see that the lifespan of $y$ is given by
\begin{align*}
T_{max}(\varepsilon,\kappa) & = -\frac{1}{2\varepsilon\kappa\oG} \ln \left(\frac{-2\kappa\oG}{\frac{3}{8}\gamma_{ppp}(0) \alpha^\prime(z)} + 1\right)\\
& \leq \varepsilon^{-1}\left(\frac{8}{3} \frac{1}{\gamma_{ppp}(0) \alpha^\prime(z)} + c \kappa \right)
\end{align*}
where the constant $c>0$ may be chosen uniformly for all $\kappa,\varepsilon \in (0,\nu]$. If $\nu$ is so small that
\begin{align}
\label{cond10}
\frac{8}{3} \frac{1}{\gamma_{ppp}(0) \alpha^\prime(z)} + c \nu < \frac{3}{\gamma_{ppp}(0) \alpha^\prime(z)} = \frac{3}{4} \oT
\end{align}
then $T_{max}(\varepsilon,\kappa) < T_\varepsilon$ and $W$ cannot be continued as a differentiable function on $[0,T_\varepsilon]$.
\end{proof}
The obvious scaling
\begin{align*}
(t,x) \mapsto ((\varepsilon \kappa)^{-1}t,(\varepsilon \kappa)^{-1}x)
\end{align*}
yields
\begin{cor}
\label{cor1}
There exist $\nu$ such that for all $\varepsilon, \kappa \in (0,\nu]$ the maximal time of existence $T^\ast>0$ for the unique solution $u \in C^1([0,T^\ast) \times \RR)$ of the Cauchy problem
\begin{align}
\label{CP2}
\left\lbrace
\begin{aligned}
u_t + a(u) u_x & =  g(u),\\ u(0,x) & = u^{0\varepsilon}((\varepsilon \kappa)^{-1}x)
\end{aligned}\right.
\end{align}
is finite. The solution $u$ satisfies
\begin{align}
\label{ubounddelta}
|u(t,x)| < \delta \; \text{for all} \; (t,x) \in [0,T^\ast) \times \RR
\end{align}
and $T^\ast$ is bounded by
\begin{align*}
T^\ast < \oT \kappa.
\end{align*}
\end{cor}
\begin{proof}
Choose $\nu \in (0,1]$ so small that the statements of Lemma \ref{lem3} and Proposition \ref{thm1} hold. By Lemma \ref{lemexist} and $\norm{u^{0\varepsilon}}{\infty} < \varepsilon c_M < \delta$ the set
\begin{align*}
E_4 := \lbrace T \in [0,\oT \kappa)| \; \exists\, u \in C^1([0,T] \times \RR): \; u \; \textnormal{solves} \; \eqref{CP2} \rbrace
\end{align*}
is non-empty, in particular there is $T^\prime > 0$ such that $[0,T^\prime] \subset E_4$. Let $T \in E_4 \backslash \lbrace 0 \rbrace$. Then $u \in C^1([0,T] \times \RR)$ and $0<T\leq \oT \kappa$. We find $u \in C^2([0,T] \times \RR)$ by Lemma \ref{lem4} and $|u(t,x)| < \varepsilon c_M < \delta$ for all $(t,x) \in [0,T] \times \RR$ by Lemma \ref{lem3}. Applying Lemma \ref{lemexist} there is a $\Delta t > 0$ such that $u$ can be extend as a $C^2$-solution of \eqref{CP2} onto $[0,T+\Delta t]$. We deduce that $E_4$ is open and of the form $E_4 = [0,T^\ast)$ for some $T^\ast \in (0,\oT \kappa]$. In summary: The unique $C^1$-solution $u$ on $[0,T^\ast)$ of \eqref{CP2} is in fact a $C^2$-solution on $[0,T^\ast)$ and satisfies $|u(t,x)| < \delta$ for all $(t,x) \in [0,T^\ast) \times \RR$. Finally, Proposition \ref{thm1} yields $T^\ast \leq \frac{3}{4} \oT \kappa < \oT \kappa$.
\end{proof}
Theorem \ref{mainthm} is an immediate consequence.
\begin{rem}
Note that these data are in general not small in $C^1$: The first derivative of the data in Corollary \ref{cor1} has its supremum bounded from below by
\begin{align*}
\kappa^{-1} \max_x \alpha^\prime \min_{|u| \leq \delta} |r_p(u)| > 0.
\end{align*}
\end{rem}

\textbf{Acknowledgements:} We want to thank H. Freistühler for proposing this problem and making helpful suggestions on how to approach it. During the research for this paper the author has been supported by Deutsche Forschungsgemeinschaft under Grant No. FR822/10-1.

\end{document}